\documentclass[11pt]{amsart}
\usepackage{amsmath,amssymb,amsfonts,amsthm, amscd,indentfirst}
\usepackage{amsmath,latexsym,amssymb,amsmath,
amscd,amsthm,amsxtra}
\usepackage{hyperref}\usepackage{url}
\usepackage{color}
\usepackage{pdfpages}
\usepackage{graphics}
\usepackage{enumitem}

\DeclareMathOperator{\arccosh}{arccosh}
\newtheorem{theorem}{Theorem}[section]
\newtheorem{prop}{Proposition}[section]

\newtheorem{corollary}{Corollary}[section]

\newtheorem{remark}{\textbf{Remark}}[section]

\def\rr{\mathbb{R}}
\def\ss{\mathbb{S}}
\def\hh{\mathbb{H}}
\def\bb{\mathbb{B}}

\def\O{\Omega}
\def\p{\partial}

\def\p{\partial}

\def\S{{\Sigma}}
\def\<{\langle}
\def\>{\rangle}
\def\div{{\rm div}}
\def\n{\nabla}
\def\G{\Gamma}
\def\ode{\bar{\Delta}}
\def\on{\bar{\nabla}}
\def\De{\Delta}

\numberwithin{equation} {section}

\begin{document}

\title[\tiny{A Minkowski type inequality with free boundary in space forms}]{A Minkowski type inequality with free boundary in space forms}
\author{Jinyu Guo}
\address{Department of Mathematics, Tsinghua University, Beijing, 100084, China}
\email{guojinyu14@163.com}


\begin{abstract}
In this paper, we consider a Minkowski inequality for a domain supported on any totally umbilical hypersurface with free boundary in space forms. We generalize the main result in \cite{Xia} into free boundary case and obtain a free boundary version of optimal weighted Minkowski inequality in space forms.
\end{abstract}

\date{}


\maketitle





\section{Introduction}
In the celebrated paper \cite{Reilly4}, Reilly applied Reilly formula \cite{Reilly2,Reilly3} to prove the following classical Minkowski inequality:
\begin{theorem}[\cite{Reilly4}]\label{Reilly-thm}\
Let $\left(\Omega^{n}, g\right)$ be a compact $n$-dimensional Riemannian manifold with smooth convex boundary $\partial\Omega=\Sigma$ and non-negative Ricci curvature. Assume $H$ is mean curvature of $\Sigma$. Then
\begin{equation}\label{Reilly}
  \operatorname{Area}(\Sigma)^{2} \geq \frac{n}{n-1} \operatorname{Vol}(\Omega) \int_{\Sigma} H d A,
\end{equation}
the equality in \eqref{Reilly} holds if and only if $\Omega$ is isometric to an Euclidean ball.
\end{theorem}

The above classical Minkowski inequality was further generalized by C. Xia in \cite{Xia}, who proved a weighted Minkowski type inequality in space forms and obtained an optimal Minkowski type inequality as follows:
\begin{theorem}[\cite{Xia}]\label{Xia}
Let $\Omega^{n} \subset \mathbb{H}^{n}\, (\,\mathbb{S}_{+}^{n}\,resp.)$ be a compact $n$-dimensional domain with smooth boundary $\Sigma$. Let $V(x)=\cosh r\,(\,\cos r\,resp.)$, where $r(x)=\operatorname{dist}(x, p)$ for some fixed point $p \in \mathbb{H}^{n}\,(p \in \mathbb{S}_{+}^{n}\, resp.)$. Assume $H$ is the mean curvature, and the second fundamental form $h$ of $\Sigma$ satisfies
\begin{equation}\label{h-ij}
  h_{\alpha\beta} \geq (\bar{\nabla}_{\nu} \log V) {g}_{\alpha\beta},
\end{equation}
where $\nu$ is the unit outward normal to $\Sigma$ and ${g}_{\alpha\beta}$ is the induced metric on $\Sigma$.
Then we have
\begin{equation}\label{space-form}
  \left(\int_{\Sigma} V d A\right)^{2} \geq \frac{n}{n-1} \int_{\Omega} V d \Omega \int_{\Sigma} H V d A,
\end{equation}
the equality in \eqref{space-form} holds if and only if $\Omega$ is a geodesic ball $B_{R}(q)$ for some point $q \in \mathbb{H}^{n}(q \in$ $\mathbb{S}_{+}^{n}$ resp.). In particular, \eqref{space-form} holds true when $\Sigma$ is horo-spherical convex in the case $\Omega \subset \mathbb{H}^{n}$ or $\Sigma$ is convex and $p \in \Omega$ in the case $\Omega \subset \mathbb{S}_{+}^{n}.$
\end{theorem}

Xia's proof is based on the solvability of the following elliptic linear equation in space forms $\mathbb{M}^{n}(K)$:
\begin{equation}\label{Re}
\begin{cases}{}
\bar{\Delta} f+Knf=1, & \text{in}\ \Omega,\\
Vf_{\nu}-V_{\nu}f=cV, & \text{on}\ \Sigma.\\
\end{cases}
\end{equation}
for  $c=\frac{\int_{\Omega} Vd\Omega}{\int_{\Sigma} VdA}$. He applied the following weighted Reilly formula \eqref{qx} in Theorem \ref{QLX}, which is first established by Qiu and Xia, to the solution of \eqref{Re} to derive
\begin{equation}\label{Mink}
 \frac{n-1}{n} \int_{\Omega}Vd\Omega \geq c^{2} \int_{\Sigma} HV d A,
\end{equation}
which is \eqref{space-form}.


\begin{theorem}[\cite{LX, QX2, WX}]\label{QLX}
Let $\O$ be a bounded domain in a Riemannian manifold $(\bar M^{n}, \bar g)$
with piecewise smooth boundary $\p\O$. Assume that $\p \Omega$ is decomposed into two smooth pieces $\p_1 \Omega$ and $\p_2 \Omega$ with a common boundary $\Gamma$. 
Let $V$ be a non-negative smooth function  on $\bar \O$ such that $\frac{\bar \n^2 V}{V}$ is continuous up to $\p\O$.
Then for any function $f\in C^{\infty}(\bar \O\setminus \G)$,
we have
\begin{eqnarray}\label{qx}
&&\int_\O V\left(\left(\ode f-\frac{\bar \De V}{V}f\right)^2-\left|\on^2 f-\frac{\bar \n^2 V}{V}f\right|^2\right)d\O\\
&=&\int_\O \left(\ode V\bar g- \on^2 V+V\cdot{\overline{Ric}}\right)\left(\bar \n f-\frac{\bar \n V}{V}f, \bar \n f-\frac{\bar \n V}{V}f\right)d\O\nonumber
 \\&&+\int_{\p\O} V\left(f_\nu-\frac{V_\nu}{V}f\right)\left(\De f-\frac{\De V}{V}f\right)dA - \int_{\p\O} Vg\left(\n \left(f_\nu-\frac{V_\nu}{V}f\right), \n f-\frac{\n V}{V}f\right)dA
 \nonumber\\&&+\int_{\p\O} VH\left(f_\nu-\frac{V_\nu}{V}f\right)^2 +\left(h-\frac{V_{\nu}}{V}g\right)\left(\n f-\frac{\n V}{V}f, \n f-\frac{\n V}{V}f\right)dA.\nonumber
\end{eqnarray}
Here $\bar{\nabla}$, $\bar{\Delta}$ and $\bar{\nabla}^{2}$ are the gradient, the Laplacian and the Hessian on $\bar{M}^{n}$ respectively, while ${\nabla}$, ${\Delta}$ and ${\nabla}^{2}$ are the gradient, the Laplacian and the Hessian on $\partial\Omega$ respectively. ${\overline{Ric}}$ is the Ricci $2$-tensor of $(\bar M, \bar g)$.
\end{theorem}

The study of free boundary surfaces or hypersurface has attracted many attention in recent decades. Inspiring results are contained in a series of papers of Fraser and Schoen \cite{FS1,FS2,FS3} about minimal hypersurfaces with free boundary in a ball $B$ and the first Steklov eigenvalue. Here  ``free boundary'' means the hypersurface intersects $\partial B$ orthogonally. There have been plenty of works about existence \cite{Jost,ST,MLi1}, regularity \cite{GHN1, MLi2}, stability  \cite{WX, Ros2, HZL2} of free boundary constant mean curvature (or minimal) hypersurfaces in a ball. We refer to a nice survey paper \cite{MLi3, WXsurvey} for more details. In this paper we are interested in obtaining new inequalities for weighted Minkowski inequalities with free boundary. We will use Theorem \ref{QLX} to generalize Theorem \ref{Xia} into the setting of domains with partial umbilical free boundary in space forms.

\section{Preliminaries}

First of all, we remark about our notations. Let $\mathbb{M}^{n}(K)$ be a complete simply-connected Riemannian manifold with constant sectional curvature $K$. Up to homotheties we may assume $K=0, 1, -1$; the case $K=0$ corresponds to the case of the Euclidean space $\mathbb{R}^{n}$, $K=1$ is the unit sphere $\mathbb{S}^{n}$ with the round metric and $K=-1$ is the hyperbolic space $\mathbb{H}^{n}$.

It is well-known that  an umbilical hypersuface in space forms has constant principal curvature. We use $S_{K, \kappa}$ to denote an umbilical hypersurface in $\mathbb{M}^{n}(K)$ with principal curvature $\kappa\in\mathbb{R}$. By a choice of orientation (or normal vector field $\bar N$), we may assume $\kappa\in [0, \infty)$.

It is also a well-known fact that in $\mathbb{R}^{n}$ and $\ss^n$,  geodesic spheres $(\kappa>0)$ and totally geodesic hyperplanes $(\kappa=0)$ are all complete umbilical hypersurfaces, while in  $\mathbb{H}^{n}$ the family of all complete umbilical hypersurfaces includes geodesic spheres $(\kappa>1)$, equidistant hypersurfaces $(0<\kappa<1)$, horospheres $(\kappa=1)$ and totally geodesic hyperplanes $(\kappa=0)$ (see e.g \cite{RAF}). Unlike geodesic spheres. Equidistant hypersurfaces, horospheres and totally geodesic hyperplanes, which are called {\it support hypersurfaces}, are all non-compact umbilical hypersurfaces in $\mathbb{H}^{n}$.

Since $S_{K, \kappa}$ divides  $\mathbb{M}^{n}(K)$ into two connected components. We use $B^{{\rm int}}_{K, \kappa}$ to denote the component whose outward normal is given by the orientation $\bar N$. For the other one, we denote by $B^{{\rm ext}}_{K, \kappa}$.

Now we state our main results in this paper as follows
\begin{theorem}\label{Guo}
Let $\O\subset B^{{\rm int}}_{K, \kappa}$ be a bounded, connected open domain whose boundary $\p\O=\bar{\Sigma}\cup T$, where $\S\subset B^{\rm int}_{K, \kappa}$ is a smooth compact hypersurface and $T\subset S_{K,\kappa}$ meets $\Sigma$ orthogonally at a common $(n-2)$-dimensional submanifold $\Gamma$. Suppose $\S$ has mean curvature $H$ and the second fundamental form $h$ satisfies
\begin{equation}\label{h-ij-2}
  h_{\alpha\beta} \geq (\bar{\nabla}_{\nu} \log V) {g}_{\alpha\beta},
\end{equation}
where $V$ is given by \eqref{E-V},\eqref{H-V} and \eqref{S-V} below.
\begin{itemize}
  \item [(i)]If $S_{K,\kappa}$ is an Euclidean plane $(\kappa=0)$ or a support hypersurface $(0\leq\kappa\leq1)$ in $\mathbb{H}^{n}$. Then
  \begin{equation}\label{space-form-2}
 \left(\int_{\Sigma} V d A\right)^{2} \geq \frac{n}{n-1} \int_{\Omega} V d \Omega \int_{\Sigma} H V d A.
\end{equation}
  \item [(ii)]If $S_{K,\kappa}$ is a geodesic sphere in $\mathbb{M}^{n}(K)$ or a totally geodesic hyperplane $(\kappa=0)$ in $\mathbb{S}^{n}$. Then the same conclusion in (i) holds provided $\Omega\subset B^{{\rm int},+}_{K, \kappa}$.
\end{itemize}
Moreover, the above equality \eqref{space-form-2} holds if and only if $\S$ is a part of an umbilical hypersurface.
\end{theorem}
\begin{remark}If $S_{K,\kappa}$ is an Euclidean plane or a support hypersurface in $\mathbb{H}^{n}$, then $V>0$, see \eqref{E-V} and \eqref{H-V} below.

If $S_{K,\kappa}$ is a geodesic sphere in $\mathbb{M}^{n}(K)$ or a totally geodesic hyperplane in $\mathbb{S}^{n}$, then $V>0$ when $\Omega\subseteq B^{{\rm int},+}_{K, \kappa}$. Here $B^{{\rm int},+}_{K, \kappa}$ means a half ball, see \eqref{half-ball-E},\eqref{half-ball},\eqref{half-ball-sphere},\eqref{half-plane} below.
\end{remark}

\begin{corollary}\label{cor-7}
Let $\O\subset B^{{\rm int}}_{K, \kappa}$ and $V$ be as above in Theorem \ref{Guo}. If $\Sigma$ is convex (horo-spherical convex resp.) supported by totally geodesic hyperplane (support hypersurfaces resp.) in $\mathbb{R}^{n}$ or $\mathbb{S}^{n}$ $(\,\mathbb{H}^{n}\, resp.)$. Then \eqref{h-ij-2}, and hence \eqref{space-form-2}, holds true.
\end{corollary}
 Next we are in a position to state the following weighted Alexandrov-Fenchel inequality with free boundary supported on an umbilical hypersurface in $\mathbb{M}^{n}(K)$, which recovers \cite[Theorem 1.9]{LX} if $\partial\Sigma$ is empty and \cite[Theorem 1.1]{CHL} for $k=1$ if $S_{K,\kappa}$ is a geodesic sphere.
\begin{theorem}\label{AF-free}
Let $\O\subset B^{{\rm int}}_{K, \kappa}$ and $V$ be as in Theorem \ref{Guo}.  Assume $\Sigma$ is sub-static, that is,
\begin{equation}\label{sub-sta}
  \Delta V g-\nabla^{2}V+V\cdot Ric\geq0.
\end{equation}
Here $\Delta$, $\nabla^{2}$ and $Ric$ are the Laplacian, Hessian and Ricci curvature on $\Sigma$ respectively.
\begin{itemize}
  \item [(i)]If $S_{K,\kappa}$ is an Euclidean plane $(\kappa=0)$ or a support hypersurface $(0\leq\kappa\leq1)$ in $\mathbb{H}^{n}$. Then
\begin{equation}\label{AF-2}
  \left(\int_{\Sigma}HVdA\right)^{2}\geq\frac{2(n-1)}{n-2}\int_{\Sigma}VdA\int_{\Sigma}\sigma_{2}(h)VdA,
\end{equation}
where $\sigma_{2}(h)=\frac{1}{2}(H^{2}-|h|^{2})$ is the second mean curvature of $\Sigma$. 
\item [(ii)]If $S_{K,\kappa}$ is a geodesic sphere in $\mathbb{M}^{n}(K)$ or a totally geodesic hyperplane $(\kappa=0)$ in $\mathbb{S}^{n}$. Then the same conclusion in (i) holds provided $\Omega\subset B^{{\rm int},+}_{K, \kappa}$.
\end{itemize}
Moreover, if there exists one point in $\Sigma$ such that \eqref{sub-sta} strictly holds, then equality \eqref{AF-2} holds if and only if $\Sigma$ is a part of an umbilical hypersurface.
\end{theorem}

\begin{corollary}\label{cor-AF-free}
Let $\O\subset B^{{\rm int}}_{K, \kappa}$ and $V$ be as in Theorem \ref{Guo}. Assume $\Sigma$ is convex and satisfies the condition \eqref{h-ij-2}. Then \eqref{AF-2} holds.

\end{corollary}
\
In the end, we follow the same argument of Theorem \ref{AF-free} to generalize De Lellis-Topping type almost Schur lemma \cite{DP, XC} into free boundary version in space forms.
\begin{theorem}\label{Almost-free}
Let $\O^{n}\subset B^{{\rm int}}_{K, \kappa}$ $(n\geq4)$ and $V$ be as in Theorem \ref{Guo}. Suppose $\Sigma$ is sub-static.
\begin{itemize}
  \item [(i)]If $S_{K,\kappa}$ is an Euclidean plane $(\kappa=0)$ or a support hypersurface $(0\leq\kappa\leq1)$ in $\mathbb{H}^{n}$. Then
\begin{equation}\label{Almost}
  \int_{\Sigma}|\mathcal{R}-\mathcal{R}^{V}|^{2}V dA\leq\frac{4(n-1)(n-2)}{(n-3)^{2}}\int_{\Sigma}\left(Ric-\frac{\mathcal{R}}{n-1}g\right)^{2} V dA,
\end{equation}
where $\mathcal{R}^{V}:=\frac{\int_{\Sigma}\mathcal{R}VdA}{\int_{\Sigma}VdA}$ and $\mathcal{R}$ is scalar curvature of $\Sigma$.
\item [(ii)]If $S_{K,\kappa}$ is a geodesic sphere in $\mathbb{M}^{n}(K)$ or a totally geodesic hyperplane $(\kappa=0)$ in $\mathbb{S}^{n}$. Then the same conclusion in (i) holds provided $\Omega\subset B^{{\rm int},+}_{K, \kappa}$.
\end{itemize}
Moreover, if there exists one point in $\Sigma$ such that \eqref{sub-sta} strictly holds, then the above equality \eqref{Almost} holds if and only if $(\Sigma, g)$ is Einstein.
\end{theorem}


\


\section{Weight functions in space forms}
In this section we will introduce the weight functions in space forms for the free boundary setting. First of all, we have the following crucial proposition that we will be used later.
\begin{prop}[\cite{LS, WengX, HWYZ}]\label{xaa2999e}
Let $\Sigma$ be a strictly convex hypersurface with free boundary in an Euclidean ball $\mathbb{B}^{n}_{r}(0)$. Then there exists a constant vector $e\in int(\widehat{\partial\Sigma})$ such that the following estimates hold on $\Sigma$:
\begin{itemize}
  \item [(1)] $\langle x,e\rangle\geq\delta_{1}$;
  \item [(2)] $\langle\nu,e\rangle\leq-\delta_{2}$;
  \item [(3)]$\langle x-e,\nu\rangle\geq\delta_{3}$;
  \item [(4)]$\langle x,\nu\rangle\leq0$,
\end{itemize}
where $x$ is the position vector in $\mathbb{R}^{n}$, $\nu$ is the unit outward normal vector of $\Sigma$, and $\delta_{i}$, $i=1, 2, 3$ are positive constants depending on $\Sigma$.
\end{prop}
\begin{proof}
See \cite[Section 4]{LS}, \cite[Proposition 2.13]{WengX} or \cite[Proposition 2.5]{HWYZ}.
\end{proof}

\


\noindent{\bf In $\mathbb{R}^{n}$ case $(K=0)$.}\
\begin{itemize}
  \item [$(i).$]If $S_{K,\kappa}$ is a sphere with radius $R$ in $\mathbb{R}^{n}$, then $\kappa=\frac{1}{R}\in(0,\infty)$. We denote $B_{K, \kappa}^{\rm int}=\left\{x\in \rr^n: |x|\le R\right\}.$
 Moreover, let
\begin{eqnarray}\label{half-ball-E}
B^{{\rm int},+}_{K, \kappa}=\left\{x\in B_{K, \kappa}^{\rm int}: \langle x,e\rangle>0\right\}
\end{eqnarray}
be an Euclidean half ball and $e$ be a fixed constant vector.

\

\

\item  [$(ii).$]If $S_{K,\kappa}$ is a hyperplane in $\mathbb{R}^{n}$, then $\kappa=0$. We denote a half-space $B_{K, \kappa}^{\rm int}$ by
\begin{equation*}
B_{K, \kappa}^{\rm int}=\left\{x\in \rr^n:  x_{n}>0\right\}.
\end{equation*}
Let $V$ be a function in $\mathbb{R}^{n}$ as follows:
 \end{itemize}
\begin{equation}\label{E-V}
 V=
\begin{cases}{}
\langle x,e\rangle, & \text{if}\ \kappa>0,\\
1, & \text{if}\ \kappa=0.\\
\end{cases}
\end{equation}

\

\noindent{\bf In $\mathbb{H}^{n}$ case $(K=-1)$.}
\begin{itemize}
   \item [$(i).$] If $S_{K,\kappa}$ is a geodesic sphere of radius $R$ in $\mathbb{H}^{n}$,  then $\kappa=\coth R\in (1,\infty)$ and let $B_{K, \kappa}^{\rm int}$ denote the geodesic ball enclosed by $S_{K, \kappa}$. By using the  Poincar\'e ball model:
\begin{eqnarray}\label{poin-ball}
\mathbb{B}^{n}=\{x\in \rr^n: |x|<1\}, \quad \bar{g}=\frac{4}{(1-|x|^{2})^{2}}\delta.
\end{eqnarray}
We have, up to an hyperbolic isometry,
  $$B_{K, \kappa}^{\rm int}=\left\{x\in \bb^n: |x|\le R_\rr:=\sqrt{\frac{1-\arccosh R}{1+\arccosh R}}\right\}.$$
 Moreover, let
\begin{eqnarray}\label{half-ball}
B^{{\rm int},+}_{K, \kappa}=\left\{x\in B_{K, \kappa}^{\rm int}: \langle x,e\rangle>0\right\}
\end{eqnarray}
 be a geodesic half ball and $e$ be a fixed constant vector in $\mathbb{R}^{n}$.
 \item [$(ii).$] If $S_{K,\kappa}$ is a support hypersurface, i.e. totally geodesic hyperplane $(\kappa=0)$, equidistant hypersurface $(0<\kappa<1)$ or horosphere $(\kappa=1)$, then $\kappa\in [0, 1]$.
By using the upper half-space model:
 \begin{eqnarray}\label{half-space}
\hh^n=\{x=(x_{1}, x_{2},\cdots,x_{n})\in \rr^n_+: x_n>0\},\quad \bar g=\frac{1}{x_n^2}\delta.
\end{eqnarray}
Let
  \begin{equation}\label{support-hy}
B_{K, \kappa}^{\rm int}=
\begin{cases}{}
\{x\in\mathbb{R}^{n}_{+}: x_n>1\},     & \text{if}\,\, \kappa=1,\\
\{x\in\mathbb{R}^{n}_{+}: x_1\tan \theta +x_{n}>1\}, & \text{if}\,\,\kappa=\cos \theta\in (0, 1),\\
\{x\in\mathbb{R}^{n}_{+}: x_1>0\}, & \text{if}\,\,\kappa=0.\\
\end{cases}
\end{equation}

\end{itemize}

\quad\quad We denote a function $V$ in $\mathbb{H}^{n}$ as follows:
\begin{equation}\label{H-V}
V=
\begin{cases}{}
\frac{2\langle x,e\rangle}{1-|x|^{2}}, & \text{if}\ \kappa>1\,\,\text{in Poincar\'{e} ball model}\,\eqref{poin-ball},\\
\frac{1}{x_{n}}, & \text{if}\ 0\leq\kappa\leq 1\,\,\text{in upper half-space model}\,\eqref{half-space}.\\
\end{cases}
\end{equation}

\vspace{0.4cm}

\noindent{\bf In $\mathbb{S}^{n}$ case $(K=1)$.}

We use the following model
\begin{equation}\label{Sphere-model}
  \left(\rr^{n},\, \bar g_{\ss}=\frac{4}{(1+|x|^2)^2}\delta\right)
\end{equation}
to represent $\ss^{n}\setminus\{{s}\}$, the unit sphere without the south pole. Let $B_R^{\ss}$ be a geodesic ball in $\ss^{n}$ with radius $R\in (0, \pi)$ centered at the north pole. The corresponding $R_\rr:=\sqrt{\frac{1-\cos R}{1+\cos R}}\in (0, \infty)$.
\begin{itemize}
   \item [$(i).$] If $S_{K,\kappa}$ is a geodesic sphere of radius $R$ in $\mathbb{S}^{n}$,  then $\kappa=\cot R\in (0,\infty)$, for $R<\frac{\pi}{2}$. In the above model
$$B^{\rm int}_{K,\kappa}=\left\{x\in \rr^n: |x|<R_{\rr}\right\}.$$
Let $B^{{\rm int},+}_{K, \kappa}$ be a geodesic half ball given by
\begin{equation}\label{half-ball-sphere}
  B^{{\rm int},+}_{K, \kappa}=\{x\in B^{{\rm int}}_{K, \kappa}:\, \langle x,e\rangle>0 \},
\end{equation}
where $e$ is a fixed constant vector in $\mathbb{R}^{n}$.

\item [$(ii).$] If $S_{K,\kappa}$ is a totally geodesic hyperplane in $\mathbb{S}^{n}$, then $\kappa=0$. In the above model $$B^{\rm int}_{K,\kappa}=\mathbb{S}^{n}_{+}:=\left\{x\in \rr^n: x_{n}>0\right\}.$$
Let $B^{{\rm int},+}_{K, \kappa}$ be given by
\begin{equation}\label{half-plane}
  B^{{\rm int},+}_{K, \kappa}=\{x\in B^{{\rm int}}_{K, \kappa}:\, |x|<1 \}.
\end{equation}
\noindent We denote a function $V$ in $\mathbb{S}^{n}$ as follows:
\begin{equation}\label{S-V}
V=
\begin{cases}{}
\frac{2\langle x,e\rangle}{1+|x|^2}, & \text{if}\ \kappa>0,\\
\frac{1-|x|^{2}}{1+|x|^2}, & \text{if}\ \kappa=0.\\
\end{cases}
\end{equation}

\end{itemize}

\begin{prop}[\cite{GX2,WX}]\label{xaa2} \,The weight function $V$ satisfies the following properties:
\
\begin{equation}\label{va2}
  \bar{\n}^2 V= -KV  \bar{g },\quad\text{in}\,\, \mathbb{M}^{n}(K),
\end{equation}
\begin{equation}\label{va3}
  \nabla_{\bar{N}}V=\kappa V, \quad\,\,\text{on}\,\, S_{K,\kappa}.
\end{equation}
where $\bar N$ is the outward unit normal of $B_{K,\kappa}^{\rm int}$.
\end{prop}
\begin{proof}
See \cite[Proposition 2.2]{GX2} and \cite[Proposition 4.2]{WX}.
\end{proof}

\section{Proof of main results}

\noindent{\bf Proof of Theorem \ref{Guo}.}\
Let $f$ be the solution to the following mixed boundary value problem:
\begin{equation}\label{Fredholm}
\begin{cases}{}
\bar{\Delta} f+nKf=1, & \text{in}\ \Omega,\\
Vf_{\nu}-V_{\nu}f=cV, & \text{on}\ \Sigma,\\
f_{\bar{N}}-\kappa f=0, & \text{on}\ T,
\end{cases}
\end{equation}
where $V_{\nu}:=\bar{\nabla}_{\nu} V$ and $c=\frac{\int_{\Omega} Vd\Omega}{\int_{\Sigma} VdA}$. Since $\Sigma$ intersects $T$ orthogonally, there exists a unique solution $f \in C^{\infty}(\bar{\Omega}\backslash\Gamma)\cap C^{1}(\bar{\Omega})$ to \eqref{Fredholm}, up to an additive $\alpha V$ for $\alpha \in \mathbb{R}$. In fact, it follows from the Fredholm alternative \cite{GT} that there exists a unique solution $w \in C^{\infty}(\bar{\Omega}\backslash\Gamma)\cap C^{1}(\bar{\Omega})$\footnote{For the
existence and the regularity of $w$, we can see in \cite{Gruter}, \cite[Section 4.1]{Liebm} or \cite[Remark 3.1]{LWW}.} (up to an additive constant) to the following Neumann equation:
\begin{equation}\label{Fredholm2}
\begin{cases}{}
\operatorname{div}\left(V^{2} \bar{\nabla} w\right)=V, & \text{in}\ \Omega,\\
V^{2} w_{\nu}=c V, & \text{on}\ \Sigma,\\
V^{2} w_{\bar{N}}=0, & \text{on}\ T,
\end{cases}
\end{equation}
if and only if $c=\frac{\int_{\Omega} Vd\Omega}{\int_{\Sigma} V dA}$. By Proposition \ref{xaa2}, one checks that $f=w V$ solves \eqref{Fredholm}.
From Proposition \ref{xaa2} again, we can see
\begin{eqnarray}\label{facts}
&&\ode V+nK V=0 ,\quad \bar \Delta V\bar g-\bar \n^2 V+V\cdot\overline{{\rm Ric}}=0.
\end{eqnarray}
Using H\"{o}lder inequality and the equation in \eqref{facts}, we derive from Theorem \ref{QLX} that
\begin{eqnarray}
\frac{n-1}{n} \int_{\Omega} V d\Omega&\geq&  \int_{\Omega} V\left((\bar{\Delta} f+n K f)^{2}-\left|\bar{\nabla}^{2} f+K f \bar{g}\right|^{2}\right)d\Omega \label{green-11}\\
&=& c\int_{\Sigma} \left(V\Delta f-\Delta Vf\right)dA+c^{2}\int_{\Sigma}VH dA\nonumber\\
&&\quad+\int_{\Sigma}\left(h-\frac{V_{\nu}}{V}g\right)\left(\nabla f-\frac{\nabla V}{V}f,\nabla f-\frac{\nabla V}{V}f\right)dA\nonumber\\
&&\quad +\int_{T}\left(h-\frac{V_{\bar{N}}}{V}g\right)\left(\nabla f-\frac{\nabla V}{V}f, \nabla f-\frac{\nabla V}{V}f\right)dA\nonumber\\
&=& c\int_{\Sigma} \left(V\Delta f-\Delta Vf\right)dA+c^{2}\int_{\Sigma}VH dA\nonumber\\
&&\quad+\int_{\Sigma}\left(h-\frac{V_{\nu}}{V}g\right)\left(\nabla f-\frac{\nabla V}{V}f,\nabla f-\frac{\nabla V}{V}f\right)dA.\nonumber
\end{eqnarray}
The last equality we used that $T\subseteq S_{K,\kappa}$ and \eqref{va3}.

Now we estimate the right hand side of \eqref{green-11}. By using the free boundary assumption $\mu=\bar{N}$ along $\partial\Sigma=\partial T$, we imply
\begin{equation}\label{free-bd}
  c\int_{\Sigma} \left(V\Delta f-\Delta Vf\right)dA=c\int_{\partial\Sigma} \left(Vf_{\mu}- V_{\mu}f\right)ds=c\int_{\partial T} \left(Vf_{\bar{N}}- V_{\bar{N}}f\right)ds=0.
\end{equation}
Here we used \eqref{va3} and \eqref{Fredholm}.

By our assumption \eqref{h-ij-2} we have
\begin{equation}\label{hij-convex}
\int_{\Sigma}\left(h-\frac{V_{\nu}}{V}g\right)\left(\nabla f-\frac{\nabla V}{V}f,\nabla f-\frac{\nabla V}{V}f\right)dA\geq0.
\end{equation}

Substituting \eqref{free-bd}-\eqref{hij-convex} into \eqref{green-11} and using $c=\frac{\int_{\Omega} Vd\Omega}{\int_{\Sigma} VdA}$, we obtain the inequality
\begin{equation}\label{inequality}
  \left(\int_{\Sigma} V d A\right)^{2} \geq \frac{n}{n-1}\int_{\Omega} V d \Omega \int_{\Sigma} H V d A .
\end{equation}

If the equality in \eqref{inequality} holds, by checking the equality in \eqref{green-11} and \eqref{hij-convex}, we get
\begin{equation}\label{Obata}
\begin{cases}{}
\bar{\nabla}_{i j}^{2} f+Kf \bar{g}_{i j}=\frac{1}{n} \bar{g}_{i j}, & \text{in}\ \Omega,\\
\left(h-\frac{V_{\nu}}{V} g\right)\left(\nabla f-\frac{\nabla V}{V}f , \nabla f-\frac{\nabla V}{V}f \right)=0 , & \text{on}\ \Sigma.
\end{cases}
\end{equation}

Next we will prove $\Sigma$ is a part of an umbilical hypersurface in $\mathbb{M}^{n}(K)$ in the following cases.

\

\noindent{\bf In Euclidean space $\mathbb{R}^{n}$.} Let $\mathcal{S}:=\{x \in \Sigma \mid h(x)>0\}$.

{\bf Case 1.} If $S_{K,\kappa}$ is a sphere in $\mathbb{R}^{n}$. Since $\Omega\subseteq B_{K,\kappa}^{\rm int,+}$ (see \eqref{half-ball-E}) and $\Sigma$ is a compact hypersurface with free boundary supported on $S_{K,\kappa}$, there exists at least one elliptic point $p$ at which every principal curvature $h_{\alpha\beta}>0$. We can always assume $\mathcal{S}$ is a connected set, otherwise one can choose a connected component of $\mathcal{S}$ at $p$, still denoted by $\mathcal{S}$. Hence $\mathcal{S}$ is a non-empty  connected open set on $\Sigma$.

For any points in $\mathcal{S}$, we apply Proposition \ref{xaa2999e} to get
\begin{equation*}
  h_{\alpha\beta}-\frac{V_{\nu}}{V}g_{\alpha\beta}>-\frac{\langle\nu,e\rangle}{\langle x,e\rangle}g_{\alpha\beta}>0.
\end{equation*}
Thus by the boundary term in \eqref{Obata} we have $f=\alpha V$ for some $\alpha\in\mathbb{R}$ on $\mathcal{S}$.

On the other hand, denote $\tilde{f}:=f-\alpha V$, by \eqref{va2} and \eqref{Obata} we get
$$
\bar{\nabla}_{i j}^{2} \tilde{f}+K\tilde{f} \bar{g}_{i j}=\frac{1}{n} \bar{g}_{i j}\,\,\text{in}\,\, \Omega.
$$
Restricting the above equation on $\mathcal{S}$, where $\tilde{f}=0$ and $\tilde{f}_{\nu}=c>0$ on $\mathcal{S}$, we have
$$
h_{\alpha\beta} \tilde{f}_{\nu}=\frac{1}{n} {g}_{\alpha\beta} \text { on } \mathcal{S},
$$
which indicates that the points in $\mathcal{S}$ are umbilical, i.e. $h_{\alpha\beta}=\frac{1}{nc}{g}_{\alpha\beta}$ on $\mathcal{S}$. Thus $\mathcal{S}$ is also a closed set. Since $\Sigma$ is connected, which implies $\Sigma=\mathcal{S}$. Applying the above argument to $\Sigma$, we obtain that $\Sigma$ is umbilical.

{\bf Case 2.} If $S_{K,\kappa}$ is a hyperplane $\{x_{n}=0\}$ in $\mathbb{R}^{n}$.

Since $\Omega\subseteq B_{K,\kappa}^{\rm int}=\{x\in\mathbb{R}^{n}|\,x_{n}>0\}$ and $\Sigma$ is a compact hypersurface with free boundary supported on $S_{K,\kappa}$, there exists at least one elliptic point $p$ such that every principal curvature $h_{\alpha\beta}>0$ at $p$. Therefore, $\mathcal{S}$ is a non-empty connected open set on $\Sigma$.

For any points in $\mathcal{S}$, we have
\begin{equation*}
  h_{\alpha\beta}-\frac{V_{\nu}}{V}g_{\alpha\beta}=h_{\alpha\beta}>0.
\end{equation*}
Thus by the boundary term in \eqref{Obata} we have $f=\alpha V$ for some $\alpha\in\mathbb{R}$ on $\mathcal{S}$. Similar to the \textbf{Case 1}, we obtain $\Sigma=\mathcal{S}$ is a part of an umbilical hypersurface.


\

\noindent{\bf In Hyperbolic space $\mathbb{H}^{n}$.}\

{\bf Case 1.} If $S_{K,\kappa}$ is a support hypersurface in $\mathbb{H}^{n}$, i.e, a totally geodesic hyperplane $(\kappa=0)$, an equidistant hypersurface $(0<\kappa=\cos\theta<1)$ or horosphere $(\kappa=1)$. Using the upper half-space model \eqref{half-space}, we denote the support hypersurface by a tilted Euclidean hyperplane $\Pi$ which meets $\partial_{\infty}\mathbb{H}^{n}$ with angle $\theta$ through a point $E_{n}=(0,0,\cdots, 1)\in\mathbb{R}^{n}_{+}$, say
 $$\Pi=\{x\in\mathbb{R}^{n}_{+}: \sin\theta x_1 +\cos\theta x_{n}=\cos\theta\}\,\, \text{for}\,\,\theta\in[0,\frac{\pi}{2}].$$

Since $\Omega\subseteq B_{K,\kappa}^{\rm int}$ (see \eqref{half-ball}) and $\Sigma$ is a compact hypersurface with free boundary supported on $S_{K,\kappa}$, there exists at least one point $p$ at which the second fundamental form $h^{\delta}$ with respect to the Euclidean metric $\delta$ is positive definite.

Let $\mathcal{S}:=\{x \in \Sigma \mid (h-\frac{V_{\nu}}{V}g)(x)>0\}$. Recall the relationship of $h^{\delta}$ and $h$ that $h-\frac{V_{\nu}}{V}g=\frac{1}{V}h^{\delta}$ from \cite[page 191]{RAF3}. Then $\mathcal{S}:=\{x \in \Sigma \mid h^{\delta}(x)>0\}$ is a non-empty connected open set on $\Sigma$.

For any points in $\mathcal{S}$, we have
\begin{equation*}
  h_{\alpha\beta}-\frac{V_{\nu}}{V}g_{\alpha\beta}>0,\,\,\,\text{on}\,\, \mathcal{S}.
\end{equation*}
From the boundary term in \eqref{Obata} we have $f=\alpha V$ for some $\alpha\in\mathbb{R}$ on $\mathcal{S}$. Similar to the \textbf{Case 1 in $\mathbb{R}^{n}$}, we obtain $h_{\alpha\beta}=\frac{1}{nc}{g}_{\alpha\beta}$ on $\mathcal{S}$. Since the totally umbilical property is conformal invariance from \cite[Proposition 10.1.2]{RAF3}, that is, $h^{\delta}_{\alpha\beta}-\frac{H^{\delta}}{n}\delta_{\alpha\beta}=V(h_{\alpha\beta}-\frac{H}{n}g_{\alpha\beta})=0$ on $\mathcal{S}$, we know that $h^{\delta}_{\alpha\beta}=c_{1}{\delta}_{\alpha\beta}$ on $\mathcal{S}$ for some constant $c_{1}$. Due to free boundary condition, we get $c_{1}>0$. By the connectivity of $\Sigma$, we obtain that $\Sigma=\mathcal{S}$ is a truncated umbilical hypersurface.

{\bf Case 2.} If $S_{K,\kappa}$ is a geodesic sphere in $\mathbb{H}^{n}$. We use Poincar\'{e} ball model \eqref{poin-ball}.

 Since $\Omega\subseteq B_{K,\kappa}^{\rm int,+}\,(\text{see}\, \eqref{half-ball})$ and $\Sigma$ is a compact hypersurface with free bondary supported on $S_{K,\kappa}$, there exists at least one elliptic point $p$ at which the second fundamental form $h^{\delta}$ with respect to the Euclidean metric $\delta$ is positive definite.

 Let $\mathcal{S}:=\{x \in \Sigma \mid h^{\delta}(x)>0\}$. Then $\mathcal{S}$ is a non-empty connected open set on $\Sigma$. By the relationship of $h^{\delta}$ and $h$, we know that
 $h=\frac{(1-|x|^{2})}{2}h^{\delta}+\langle x,\nu_{\delta}\rangle g$. It follows that
 \begin{equation*}
   h_{\alpha\beta}-\frac{V_{\nu}}{V}g_{\alpha\beta}=\frac{(1-|x|^{2})}{2}h^{\delta}-\frac{(1-|x|^{2})\langle\nu_{\delta},e\rangle}{2\langle x,e\rangle}g>0, \,\,\,\text{on}\,\,\mathcal{S},
 \end{equation*}
where we used Proposition \ref{xaa2999e}. Thus by the boundary term in \eqref{Obata} we have $f=\alpha V$ for some $\alpha\in\mathbb{R}$ on $\mathcal{S}$. Similar to the above \textbf{Case 1 in $\mathbb{H}^{n}$}, we see $\Sigma=\mathcal{S}$ is a truncated umbilical hypersurface.

\

\noindent{\bf In unit Sphere $\mathbb{S}^{n}$.} We use sphere model \eqref{Sphere-model}. Thus $S_{K,\kappa}$ is a geodesic sphere or a totally geodesic hyperplane in $\mathbb{S}^{n}$. Let $\mathcal{S}:=\{x \in \Sigma \mid h^{\delta}(x)>0\}$. The proof is similar to the \textbf{Case 2 in $\mathbb{H}^{n}$}, we leave it to the interested readers.

In conclusion, we obtain that $\Sigma$ is a truncated umbilical hypersurface in $\mathbb{M}^{n}(K)$. The proof of Theorem \ref{Guo} is completed.

\qed
\

\
\

\noindent{\bf Proof of Theorem \ref{AF-free}.}\
The proof is close to J. Li and C. Xia \cite{LX}, small difference arises from the boundary computation. We prove it here for the reader's convenience.

By an algebraic computation, the inequality \eqref{AF-2} is equal to the following inequality
\begin{equation}\label{AF-thm}
  \int_{\Sigma}\left(H-\overline{H}^{V}\right)^{2}V dA\leq\frac{n-1}{n-2}\int_{\Sigma}\left(h_{\alpha\beta}-\frac{H}{n-1}g_{\alpha\beta}\right)^{2}V dA,
\end{equation}
where $\overline{H}^{V}=\frac{\int_{\Sigma}HVdA}{\int_{\Sigma}VdA}$. Therefore, we only need to prove the inequality \eqref{AF-thm}.\\

Let $f\in C^{\infty}(\bar{\Sigma})$ be a solution of
\begin{equation}\label{AF-f}
\begin{cases}{}
\Delta f-\frac{\Delta V}{V}f=H-\overline{H}^{V}, & \text{on}\ \Sigma,\\
\nabla_{\mu}f-\kappa f=0, & \text{on}\ \partial\Sigma=\partial T.
\end{cases}
\end{equation}
Here $\nabla$ and $\Delta$ are the gradient and Laplacian of $\Sigma$ respectively,  $\mu$ is conormal vector of $\Sigma$ which is equal to $\bar{N}$ along $\partial\Sigma$ due to free boundary assumption. In fact, it follows from the Fredholm alternative \cite{GT} that there exists a unique solution $\omega\in C^{\infty}(\bar{\Sigma})$ (up to an additive constant) to the following Neumann equation:
\begin{equation*}
\begin{cases}{}
\div (V^{2}\nabla\omega)=(H-\overline{H}^{V})V, & \text{on}\ \Sigma,\\
\nabla_{\mu}\omega=0, & \text{on}\ \partial\Sigma=\partial T.
\end{cases}
\end{equation*}
By \eqref{va3}, we see that $f=\omega V$ solve \eqref{AF-f}.\\

From Proposition \ref{xaa2}, we have
 \begin{equation}\label{free-boundary}
   V\nabla_{\mu}f-f\nabla_{\mu}V= V\nabla_{\bar{N}}f-f\nabla_{\bar{N}}V=V(\nabla_{\bar{N}}f-\kappa f)=0\,\,\,\text{along}\,\,\,\partial\Sigma.
 \end{equation}

For our convenience. Denote
 \begin{equation}
   A_{\alpha \beta}=\nabla_{\alpha \beta}^{2} f-\frac{\nabla_{\alpha \beta}^{2} V}{V} f\quad;\,\, A=\Delta f-\frac{\Delta V}{V} f,
 \end{equation}
 and
\begin{equation}\label{A-ring}
 \mathring{A}_{\alpha \beta}=A_{\alpha \beta}-\frac{1}{n-1} A g_{\alpha \beta}\quad;\,\,\mathring{h}_{\alpha \beta}=h_{\alpha \beta}-\frac{H}{n-1} g_{\alpha \beta}.
\end{equation}
By \eqref{AF-f}-\eqref{free-boundary} and integration by parts, we get
\begin{eqnarray}\label{AF-holder}
\int_{\Sigma} \left(H-\bar{H}^{V}\right)^{2}V dA&&=\int_{\Sigma}\left(H-\bar{H}^{V}\right)(V \Delta f-f \Delta V)\, dA\nonumber \\
&&=-\int_{\Sigma}\langle\nabla H, V \nabla f-f \nabla V\rangle \,dA+\int_{\partial\Sigma} (H-\overline{H}^{V})(V\nabla_{\mu}f-f\nabla_{\mu}V)\, ds\nonumber\\
&&=-\int_{\Sigma}\langle\nabla H, V \nabla f-f \nabla V\rangle\, dA.
\end{eqnarray}
By Codazzi property of $h_{\alpha \beta}$ in space forms, we know that
\begin{equation*}\label{Codazzi}
\nabla_{\alpha}\mathring{h}_{\alpha\beta}=\nabla_{\alpha}\left(h_{\alpha \beta}-\frac{H}{n-1} g_{\alpha \beta}\right)=\nabla_{\beta}H-\frac{1}{n-1}\nabla_{\beta}H=\frac{n-2}{n-1}\nabla_{\beta}H.
\end{equation*}
It follows that
\begin{eqnarray}
&&-\int_{\Sigma}\langle\nabla H, V \nabla f-f \nabla V\rangle\, dA\label{AF-holder}\\
&&=-\frac{n-1}{n-2} \int_{\Sigma}(\nabla_{\alpha}\mathring{h}_{\alpha \beta})\left(V \nabla_{\beta} f-f \nabla_{\beta} V\right) dA\nonumber\\
&&=\frac{n-1}{n-2} \int_{\Sigma} V\mathring{h}_{\alpha\beta}\mathring{A}_{\alpha \beta} dA-\frac{n-1}{n-2}\int_{\partial\Sigma}\left(h-\frac{H}{n-1}g\right)(V\nabla f-f\nabla V, \mu)\,ds\nonumber\\
&&=\frac{n-1}{n-2} \int_{\Sigma} V\mathring{h}_{\alpha\beta}\mathring{A}_{\alpha \beta} dA-\frac{n-1}{n-2}\int_{\partial\Sigma}\left(h(\mu,\mu)-\frac{H}{n-1}\right)(V\nabla_{\mu} f-f\nabla_{\mu} V)\,ds\nonumber\\
&&=\frac{n-1}{n-2} \int_{\Sigma} V\mathring{h}_{\alpha\beta} \mathring{A}_{\alpha \beta}\, dA\nonumber\\
&&\leq \frac{n-1}{n-2}\left(\int_{\Sigma} V|\mathring{h}_{\alpha \beta}|^{2}dA\right)^{\frac{1}{2}}\left(\int_{\Sigma} V\left|\mathring{A}_{\alpha \beta}\right|^{2}dA\right)^{\frac{1}{2}}.\nonumber
\end{eqnarray}
Here we use the fact $h(e, \mu)=0$ for any $e\in T(\p \Sigma)$. Namely, $\mu$ is a principal direction of $\partial\Sigma$ in $\Sigma$ due to $\partial\Sigma\subseteq S_{K,\kappa}$ (see \cite[Proposition 2.1]{WX}). In the last inequality we applied the H\"{o}lder inequality.

Since $\Sigma$ is an $(n-1)$-dimensional manifold with boundary $\partial\Sigma$, we use Theorem \ref{QLX} for $\Omega=\Sigma$ to get
\begin{equation*}
  \int_{\Sigma} V\left(\Delta f-\frac{\Delta V}{V} f\right)^{2}dA \geq \int_{\Sigma} V\left(\nabla^{2}f-\frac{\nabla^{2}V}{V} f\right)^{2}dA,
\end{equation*}
where we use sub-static condition \eqref{sub-sta} and boundary condition \eqref{va3}, \eqref{free-boundary}. It implies that
\begin{eqnarray}\label{AF-yu}
 \int_{\Sigma} V\left|\mathring{A}_{\alpha \beta}\right|^{2}dA&=&\int_{\Sigma} V\left(\nabla^{2}f-\frac{\nabla^{2}V}{V} f\right)^{2}dA-\frac{1}{n-1}\int_{\Sigma}V\left(\Delta f-\frac{\Delta V}{V}f\right)^{2} dA\\
 &\leq& \frac{n-2}{n-1} \int_{\Sigma} V\left(\Delta f-\frac{\Delta V}{V} f\right)^{2}dA.\nonumber
\end{eqnarray}
Combining \eqref{AF-holder}-\eqref{AF-yu} with equation \eqref{AF-f}, we conclude \eqref{AF-thm} holds.
\vspace{0.1cm}


If the equality in \eqref{AF-thm} holds, by checking the above proof, we see
\begin{equation}\label{hab}
  \mathring{h}_{\alpha \beta}=h_{\alpha \beta}-\frac{H}{n-1} g_{\alpha \beta}=\lambda\mathring{A}_{\alpha \beta}\quad \text{for some constant}\,\,\lambda.
\end{equation}
and
\begin{equation}\label{rrr}
  \left(\Delta V g-\nabla^{2} V+V {Ric}\right)\left(\nabla \frac{f}{V}, \nabla \frac{f}{V}\right)=0.
\end{equation}

Since there exists at least one point $p$ such that $\Delta V g-\nabla^{2} V+V {Ric}$ is positive definite at $p$. From continuity we know that $\Delta V g-\nabla^{2} V+V {Ric}$ is positive definite in a neighborhood $\mathcal{N}_{p}$ of $p$. From the equation \eqref{rrr}, we get $f=\eta V$ for some constant $\eta\in\mathbb{R}$ in $\mathcal{N}_{p}$.

According to \eqref{A-ring} and \eqref{hab}, we have $h_{\alpha \beta}=\frac{H}{n-1} g_{\alpha \beta}=c g_{\alpha \beta}$ in ${\mathcal{\overline{N}}_{p}}$ for some constant $c>0$.
By the proof of Theorem \ref{Guo}, we get $h-\frac{V_{\nu}}{V}g>0$ in ${\mathcal{\overline{N}}_{p}}$. It follows that $\Delta V g-\nabla^{2} V+V {Ric}=(n-2) c\left(c V-\nabla_{\nu}V\right) g>0$ in ${\mathcal{\overline{N}}_{p}}$. Therefore $\mathcal{S}_{1}:=\left\{x \in \Sigma\mid\Delta V g-\nabla^{2} V+V {Ric}>0\right\}$ is non-empty open and closed set. Thus $\mathcal{S}_{1}=\Sigma$. Repeat the argument above to $\Sigma$, we obtain that $\Sigma$ is umbilical.

\qed
\

\noindent{\bf Proof of Corollary \ref{cor-AF-free}.}\
From the Gauss equation and Proposition \ref{xaa2}, we have
\begin{eqnarray*}
\Delta V&=&\bar{\Delta} V-\bar{\nabla}^{2}_{\nu \nu}V-H \nabla_{\nu}V=-(n-1)K V-H \nabla_{\nu}V,\\
\nabla^{2}_{\alpha \beta} V&=&\bar{\nabla}^{2}_{\alpha \beta} V-(\nabla_{\nu} V) h_{\alpha \beta}=-K V g_{\alpha \beta}-(\nabla_{\nu} V) h_{\alpha \beta}, \\
Ric_{\alpha \beta}&=&(n-2)K g_{\alpha \beta}+H h_{\alpha \beta}-h_{\alpha \gamma} h_{\beta \gamma}.
\end{eqnarray*}
Here $Ric_{\alpha \beta}$ is the Ricci tensor of $\Sigma$. Thus by our convexity hypothesis, we get
\begin{equation}\label{hhhhh}
  \Delta V g_{\alpha \beta}-\nabla^{2}_{\alpha \beta}V+V Ric_{\alpha \beta}=\left(V h_{\beta \gamma}-\nabla_{\nu} V g_{\beta \gamma}\right)\left(H g_{\alpha \gamma}-h_{\alpha \gamma}\right)\geq0.
\end{equation}

Since $\Sigma$ is a compact hypersurface with free boundary supported on $S_{K,\kappa}$. By the maximum principle and inequality \eqref{hhhhh}, there exists at least one point $p$ at which $\Delta V g-\nabla^{2} V+V {Ric}$ is positive definite.
\

\qed

\

\noindent{\bf Proof of Theorem \ref{Almost-free}.}\ Following the above method of Theorem \ref{AF-free}, we replace the second fundamental form $h$ with Ricci curvature tensor $Ric$ of $\Sigma$ in \eqref{A-ring}. From Gauss equation and Codazzi equation in space forms, we know that $\nabla_{\beta}\mathcal{R}=2\nabla_{\alpha}(Ric_{\alpha\beta})$ on $\Sigma$. On the other hand, $\mu$ is a principal direction of $\partial\Sigma$ in $\Sigma$ since $S_{K,\kappa}$ is umbilical (see \cite[Proposition 2.1]{WX}), which implies $Ric(e, \mu)=0$ for any $e\in T(\p \Sigma)$ along $\partial\Sigma$. The proof is completed. \qed

\

\




{\bf Acknowledgements.} The paper was carried out while I was visiting the Mathematical Institute, the University of Freiburg under the support by the China Scholarship Council. I would like to thank the department for its hospitality. I also thank Professor Guofang Wang and Professor Chao Xia for useful comments and for their constant supports.

\

\end{document}